\documentclass[pdflatex,sn-mathphys-num]{sn-jnl}% Math and Physical Sciences Numbered Reference Style
\usepackage{graphicx}%
\usepackage{multirow}%
\usepackage{amsmath,amssymb,amsfonts}%
\usepackage{amsthm}%
\usepackage{thmtools}
\usepackage{thm-restate}
\usepackage{mathrsfs}%
\usepackage[title]{appendix}%
\usepackage{xcolor}%
\usepackage{textcomp}%
\usepackage{manyfoot}%
\usepackage{booktabs}%
\usepackage{algorithm}%
\usepackage{algorithmicx}%
\usepackage{algpseudocode}%
\usepackage{listings}%
\usepackage[T1]{fontenc}
\usepackage{amsmath, amssymb, stmaryrd, calligra}
\usepackage[english]{babel}
\usepackage{mathrsfs}
\usepackage[all]{xy}
\usepackage{enumitem}
\usepackage{url}
\usepackage{todonotes}
\usepackage{tikz-cd}

\newtheorem{thm}{Theorem}[section]
\newtheorem*{thm*}{Theorem}
\newtheorem*{cor*}{Corollary}
\newtheorem*{lem*}{Lemma}
\newtheorem{lem}[thm]{Lemma}

\newtheorem{prop}[thm]{Proposition}
\newtheorem*{prop*}{Proposition}

\theoremstyle{definition}
\newtheorem{defn}[thm]{Definition}

\newtheorem{remark}[thm]{Remark}

\newcommand{\rest}[1]{|_{\thinspace #1}}
\newcommand{\Spec}{\text{Spec }}

\newcommand{\BP}{\mathbb{P}}

\DeclareMathOperator{\shom}{\mathscr{H}\text{\kern -3pt {\calligra\large om}}\,}

\newcommand\secref[1]{\hyperref[#1]{\S\ref*{#1}}}

\def \Hom{\operatorname{Hom}}

\theoremstyle{thmstyleone}%
\theoremstyle{thmstyletwo}%

\theoremstyle{thmstylethree}%

\begin{document}

\title[Article Title]{The Second Vanishing Theorem in Ramified Mixed Characteristic}


\author*{\fnm{Alex} \sur{Scheffelin}}\email{ams2637@columbia.edu}

\affil{\orgdiv{Department of Mathematics}, \orgname{Columbia University}, \orgaddress{\street{Room 408, MC 4406
2990 Broadway}, \city{New York}, \postcode{10027}, \state{NY}, \country{United States of America}}}


\abstract{Local Cohomology, since its introduction, has served as an important invariant for commutative rings and their modules. They furthermore provide the local model for relative cohomology groups for schemes. As with all cohomology theories, vanishing theorems are widely sought after, and for local cohomology a classical theorem of Grothendieck states that all local cohomology vanishes past the dimension of the ring. Hartshorne-Lichtenbaum vanishing tells us when local cohomology vanishes at the dimension of the ring, and for vanishing one below the dimension of the ring we arrive at the Second Vanishing Theorem. This paper proves the Second Vanishing Theorem in the final unknown case for regular local rings, that being the case of ramified mixed characteristic rings, and gives a few applications of this result. The method of this paper works in equicharacteristic, and we show how we can reduce the unramified case to the ramified case as well, yielding a unified proof of the Second Vanishing Theorem in all characteristics.}




\maketitle
\bmhead{Acknowledgements}

We would like to thank many people for various discussions and feedback. Included among these are Vidhu Adhihetty, Marty Bishop, Jessie Loucks-Tavitas, Brendan Murphy, S\'andor Kov\'acs, Karl Schwede, Shinnosuke Ishiro, Thomas Lawrence, Ari Krishna, Jananan Arulseelan, Linquan Ma, and Jon Kim. Special thanks to Brian Nugent for pointing us to the result of \cite{lipmanRationalSingularitiesApplications1969} which plays a crucial role in the proof of \S6. This work was completed while studying at the University of Tokyo under the care and supervision of Kazuma Shimomoto and Shunsuke Takagi, whom the author thanks. Finally, we would like to give our utmost thanks to our advisor Johan de Jong, without whom this paper would certainly have been impossible.
\section{Introduction}
Given a ring $A$ and ideal $I$ we define the cohomological dimension of $A$ with respect to $I$ as $\mathrm{cd}(A,I) = \sup\{i|\exists M\in\text{Mod}(A),H^i_I(M)\neq 0\}$, and similarly for a scheme $X$ we define the cohomological dimension of $X,\ \mathrm{cd}(X) = \sup\{i|\exists\mathscr{F}\in\mathrm{QCoh}(X), H^i(X,\mathscr{F})\neq0\}$.

Given an open set $U\subseteq \BP^n_A$, writing $U^c = V(I)$ for a homogeneous ideal $I\subseteq A[x_0,\dots,x_n]$, for any quasicoherent $\mathcal{O}_U$-module $\widetilde{M}$ we obtain an exact sequence
\[0\to H^0_I(M)\to M\to H^0(U,\widetilde{M})\to H^1_I(M)\to 0,\]
as well as isomorphisms
\[\bigoplus^\nu H^i(U,\widetilde{M}(\nu))\cong H^{i+1}_I(M)\] where $M$ is the associated $A[x_0,\dots,x_n]$-module.

Via the isomorphism described above we see that
\[\mathrm{cd}(A[x_0,\dots,x_n],I) \leq \dim(A[x_0,\dots,x_n]) - c \Leftrightarrow \mathrm{cd}(U)\leq \dim(\BP^n_A) - c\]
whenever $r < \dim(A) + n$. Motivated by this, a natural question is to ask when $\mathrm{cd}(R,I) \leq \dim(R) - c$ for various values of $c$.

When $c = 1$ this is the content of the Hartshorne-Lichtenbaum vanishing theorem \cite[Theorem 3.1]{hartshorneCohomologicalDimensionAlgebraic1968}, which states that
\begin{thm}
    Given an $n$-dimensional Noetherian ring $A$ with an ideal $I$, the following are equivalent.
    \begin{enumerate}
        \item $\mathrm{cd}(A,I) \leq n - 1$
        \item For all maximal ideals $\mathfrak{m}$ of $A$, either $\dim A_\mathfrak{m} < n$, or
        \[\dim \widehat{A}_\mathfrak{m}/(I\widehat{A}_\mathfrak{m} + \mathfrak{p})\geq 1\]
        for all minimal primes $\mathfrak{p}$ of $\widehat{A}$ where $\dim \widehat{A}/\mathfrak{p} = n$.
    \end{enumerate}
\end{thm}
To extend this to $c = 2$, we focus our attention on regular local rings, in which case Hartshorne later poses a problem.
\begin{defn}
    Let $A$ be a regular local ring, and $I$ an ideal of $A$. Denote by $A^{\dagger}$ the ring $((\widehat{A})^{\mathrm{sh}})^{\widehat{}}$ which is the completion of the strict henselization of the completion of $A$. Then we say that the Second Vanishing Theorem holds for the pair $(A,I)$ if the following are equivalent:
    \begin{enumerate}
        \item $\mathrm{cd}(A,I)\leq \dim A - 2$
        \item $\dim A/I\geq 2$ and $\Spec(A^{\dagger}/IA^{\dagger})\setminus\{\mathfrak{m}\}$ is connected
    \end{enumerate}
\end{defn}
Progress on the Second Vanishing Theorem has proceeded steadily. After being posed in 1968, Ogus proved it in characteristic 0 \cite[Corollary 2.11]{ogusLocalCohomologicalDimension1973}, while Peskine and Szpiro proved it in positive characteristic \cite[Corollaire III 5.5]{peskineDimensionProjectiveFinie1973}. Later Huneke and Lyubeznik proved it in equicharacteristic ambivalent to the characteristic \cite{hunekeVanishingLocalCohomology1990}. Recently, Zhang proved it in unramified mixed characteristic \cite{zhangSecondVanishingTheorem2025}. In ramified mixed characterstic results are limited, Bhattacharya proved it when $I$ is extended from an unramified subring which $A$ is finite over \cite{bhattacharyyaNoteSecondVanishing2020}, and Asgharazadeh-Ishiro-Shimomoto proved it subject to additional conditions \cite{asgharzadehSurjectivityLocalCohomology2023}.

Recently Linquan Ma created an example of a complete regular local ring $R$ and an ideal $J$ such that $H^6_J(R)$ has infinitely many associated primes \cite{maInfinitelyManyAssociated2026}. This was a surprising result, as it was conjectured that for all regular local rings $A$, ideals $I$, and natural numbers $i$, that $H^i_I(A)$ has finitely many associated primes. This was also known in every case except for the ramified case, and so it seemed to suggest that local cohomological results for ramified regular local rings behave differently than in all other cases, which may have cast the validity of the Second Vanishing Theorem for ramified regular rings in question.

The main result of our paper is the following:
\begin{restatable}{thm}{maintheorem}\label{maintheorem}
    Let $A$ be a ramified regular local ring of mixed characteristic, and $I$ an ideal of $A$. Then SVT$(A,I)$ holds.
\end{restatable}
While the origins and initial work towards the problem were rooted in geometry, the efforts towards it drifted further and further algebraic as time carried on. Our method of proof returns to the geometric roots by casting it as a problem on lifting certain formal functions over a closed point. Here we crucially rely on a result of Gabber on the global generation of a normal bundle for subschemes of projective space which uses the assumption that we are in ramified mixed characteristic. This represents a departure from previous approaches to the problem.

We would especially like to draw attention to recent work by Batavia \cite{bataviaVanishingLocalCohomology2026} of a stronger result on the vanishing of local cohomology modules, which is a mixed characteristic analogue of results from \cite{hunekeVanishingLocalCohomology1990}, as well as Zhang's proof of the SVT in unramified mixed characteristic \cite{zhangSecondVanishingTheorem2025}. Both of these results work only in unramified mixed characteristic, and just like in \cite{hunekeVanishingLocalCohomology1990} rely on being able to write a complete regular local ring as a power series ring over a very nice ring (a field in equicharacteristic, or a particularly nice DVR in the unramified case). In the case where the ring is ramified, this no longer is the case; although you are able to write a ramified complete regular ring as an Eisenstein extension of an unramified subring which is a power series ring over a DVR, this is not enough to apply the same methods. As appears to often be the case, this seems to be the main issue that has stalled progress on the ramified case so far, but using the result of Gabber we are able to prove the result via a geometric method only in the ramified case. This seems to strongly mirror the story for non-negativity of intersection multiplicity for which Gabber originally proved his result on global generation, and so the methods of this paper may be adaptable to prove other ring-theoretic results which are only known to be true in the unramified case.  In this case, we believe it may be helpful to consider the ramification as a hypothesis one can leverage rather than a defect.

Our proof proceeds in several steps by making a series of reductions. First we will recast the problem in terms of formal schemes, and using this we will reduce the problem to when $A/I$ is a two-dimensional integral domain. This is not novel, but we believe our approach is more intuitive than previous approaches owing to being stated in the language of formal schemes and will serve as a warmup to the methods used later on. We then reduce to the case where $A$ has an algebraically closed residue field. While we could not find this result in the literature we suspect it was known to experts. We then reduce to the case where $A/I$ is normal and thus has an isolated singularity; this represents the first genuine departure from known results and crucially uses our geometric framework. We then conclude by taking a resolution of $A/I$ embedded in a projective space over $A$, and use cohomological arguments to argue that the required formal functions extend to regular functions on $\BP^n_A$, which are simply elements of $A$ concluding our proof. This will be the only time that we use the hypothesis that we are in ramified mixed characteristic.

\section{Formal Schemes}
We will refer to the Second Vanishing Theorem as SVT, and we will say that SVT$(A,I)$ holds if the Second Vanishing Theorem is true for the pair $A,I$.

First we note that as $A^{\dagger}$ is a faithfully flat extension of $A$ of the same dimension, and since $H^i_I(M)\otimes A^{\dagger} = H^i_{IA^{\dagger}}(M\otimes A^{\dagger})$, by base change we can immediately assume that $A$ is complete, with a separably closed residue field.

That 1. implies 2. in our statement of the SVT is easy, c.f \cite[Theorem 1.4]{zhangSecondVanishingTheorem2025}. So, we focus on proving the converse.

Given a scheme $X$ and a closed subscheme $Z$, we denote by $X_{/Z}$ the formal scheme obtained by completing $X$ along $Z$. In short, the topological space of $X_{/Z}$ is just $Z$ while its structure sheaf is the sheaf $\varprojlim \mathcal{O}_X/I_Z^n$. We now prove various generalities on formal schemes and completions. These results are not needed in this generality for this paper.

\begin{lem}\label{completions}
    Let $A$ be a Noetherian ring, and $I \subseteq J$ ideals of $A$. Then $(\widehat{A}^I)^{\widehat{}_J} = \widehat{A^J}$, and $(A/I)^{\widehat{}_{J/I}} = \widehat{A^J}/I\widehat{A^J}$.
\end{lem}
\begin{proof}
    By definition, $(\widehat{A}^I)^{\widehat{}_J} = \varprojlim \widehat{A^I}/J^n\widehat{A^J}$. Note that $\widehat{A^I}$ is a flat $A$-algebra, and so the have exact sequences
    \[0\to J^n\to A\to A/J^n\to 0\]
    give us exact sequences
    \[0\to J^n\widehat{A^I}\to \widehat{A^I}\to \left(A/J^n\right)^{\widehat{}_I}\to 0.\]
    Note that definitionally $\left(A/J^n\right)^{\widehat{}_I} = \varprojlim_m (A/J^n+I^m)$, which as $I \subseteq J$, for $m \geq n$ is equal to $A/J^n$ and thus it is just equal to $A/J^n$. This means that
    \[(\widehat{A}^I)^{\widehat{}_J} = \varprojlim \widehat{A^I}/J^n\widehat{A^J} = \varprojlim\left(A/J^n\right)^{\widehat{}_I} = \varprojlim A/J^n = \widehat{A^J}\]
    as desired.

    Note that $(A/I)^{\widehat{}_{J/I}} = (A/I)^{\widehat{}_{J}}$, and so we can look at the exact sequence
    \[0\to I\to A\to A/I\to 0\]
    to obtain, using flatness of $\widehat{A^J}$,
    \[0\to I\widehat{A^J}\to \widehat{A^J}\to (A/I)^{\widehat{}_J}\to 0\]
    so that $(A/I)^{\widehat{}_{J/I}} = \widehat{A^J}/I\widehat{A^J}$ as desired.
\end{proof}

\begin{lem}\label{restrict-complete}
    Let $\mathcal{F}$ be a coherent $\mathcal{O}_X$-module for $X$ a (locally) Noetherian scheme. Let $\mathcal{F}_{/Y}$ denote the formal completion of a coherent sheaf $\mathcal{F}$ on $X$ along the closed subset $Y$. Let $Z < W < X$ for $Z,W$ closed subschemes. Then
    \[(\mathcal{F}\rest{W})_{/Z} = (\mathcal{F}_{/W})\rest{W_{/Z}}.\]
\end{lem}
\begin{proof}
    This is a local question, so we may assume that the formal schemes are affine. We then need to show the following statement, given a Noetherian ring $A$, with ideals $I < J$, and finite module $M$, we have to show
    \[(M/IM)^{\widehat{_J}} = \widehat{M^J}\otimes_{\widehat{A^J}} (A/I)^{\widehat{}_{J/I}}.\]
    By use of Lemma \ref{completions} below we have that the left hand side is equal to
    \[M\otimes_A A/I\otimes_A \widehat{A^J} = M\otimes_A \widehat{A^J}/I\widehat{A^J} = M\otimes_A (A/I)^{\widehat{}_{J/I}},\]
    which the right hand side also reduces to.
\end{proof}

\begin{lem}\label{sheaf thickenings}
    Let $X$ be a (locally) Noetherian scheme, and $Z,W$ closed subschemes of $X$. Let $W_n$ denote the $n$-th infinitesimal neighborhood of $W$ in $X$. Then, considering these as sheaves on $X$, the system of sheaves $\mathcal{O}_{(W_n)_{/W_n\cap Z}}$ forms an inverse system such that $\varprojlim \mathcal{O}_{(W_n)_{/W_n\cap Z}} = \mathcal{O}_{X_{/W\cap Z}}$.
\end{lem}
\begin{proof}
    That the sheaves $\mathcal{O}_{(W_n)_{/W_n\cap Z}}$ form an inverse system of sheaves on $X$ is obvious, and so we may take the limit. Let $J$ be the ideal sheaf of $W$ and $I$ the ideal sheaf of $Z$. Expanding the definition of these sheaves, we end up with
    \begin{align*}
    \varprojlim_n \mathcal{O}_{(W_n)_{/W_n\cap Z}} &= \varprojlim_n(\varprojlim_m\frac{\mathcal{O}_X/J^n}{(J^n + I)^m(\mathcal{O}_X/J^n)}\\
    & = \varprojlim_n(\varprojlim_m\frac{\mathcal{O}_X/J^n}{(J^n + I^m)(\mathcal{O}_X/J^n)}\\
    &= \varprojlim_n(\varprojlim_m\ \mathcal{O}_X/(J^n + I^m))\\
    &= \varprojlim_n\ \mathcal{O}_X/(J^n + I^n)\\
    &= \varprojlim_n\ \mathcal{O}_X/(J + I)^n = \mathcal{O}_{X_{/W\cap Z}}
    \end{align*}
\end{proof}

\begin{lem}\label{integralcomplete}
    Let $A$ be an excellent normal local ring, then the completion of $A$ with respect to any ideal $I$ is an integral domain.
\end{lem}
\begin{proof}
    As was mentioned before, the $\mathfrak{m}$-adic completion of $A$ is again normal, and is thus a domain. The $\mathfrak{m}$-adic completion of the $I$-adic completion of $A$ is just the $\mathfrak{m}$-adic completion of $A$ by Lemma \ref{completions}, and the map from the $I$-adic completion into this ring is an injection. Therefore, the $I$-adic completion of $A$ is an integral domain.
\end{proof}

\begin{lem}\label{formallocalcomplete}
    Let $X$ be a (locally) Noetherian scheme, and $Z$ a closed subscheme, and $x\in Z$. Then $\widehat{(\mathcal{O}_{X_{/Z}})_x}\cong \widehat{\mathcal{O}_{X,x}}$.
\end{lem}
\begin{proof}
    Let us work locally at $x$ so that $X = \Spec A$, $Z = V(I)$, and $x$ corresponds to a prime ideal $\mathfrak{p}$. Then,
    \[(\mathcal{O}_{X_{/Z}})_x = \varinjlim_{f\notin \mathfrak{p}}(\varprojlim_nA_f/I^nA_f).\]
    Let us actually see what the $J$-adic completion of this ring is for $J > I$:
    \[\varprojlim_m[\varinjlim_{f\notin \mathfrak{p}}(\varprojlim_nA_f/I^nA_f)\otimes A/J^m].\]
    As direct limits commute with tensor product, and $A/J^m$ is finitely presented we may bring the tensor product inside to arrive at
    \[\varprojlim_m[\varinjlim_{f\notin \mathfrak{p}}(\varprojlim_nA_f/(J^m + I^n)A_f)] = \varprojlim_m[\varinjlim_{f\notin \mathfrak{p}}(\varprojlim_nA_f/J^mA_f)] = \varprojlim_m[\varinjlim_{f\notin \mathfrak{p}}A_f/J^mA_f],\]
    as $J^m + I^n = J^m$ once $n \geq m$. But we note that the inner colimit simply returns the localization of $A/J^m$ at $\mathfrak{p}$, and so we get
    \[\varprojlim_m (A/J^m)_\mathfrak{p} = \varprojlim A_\mathfrak{p}/J^mA_\mathfrak{p} = \widehat{A_\mathfrak{p}^J}.\]
    Now simply take $J$ to be $\mathfrak{p}$ and we are done.
\end{proof}

\begin{lem}\label{formalgenerization}
    Let $X$ be an excellent normal scheme, and $Z$ a closed subscheme. Then, given a specialization $x\rightsquigarrow y$, the generization map $(\mathcal{O}_{X_{/Z}})_y\to (\mathcal{O}_{X_{/Z}})_x$ is injective.
\end{lem}
\begin{proof}
    As this is local, we may assume that $X = \Spec A$ and $Z = V(I)$, and $x,y$ corresponding to prime ideals $\mathfrak{p},\mathfrak{q}$ respectively.
    By \cite[EGA 0 Proposition 7.6.17]{grothendieckElementsGeometrieAlgebrique1960} the ring $(\mathcal{O}_{X_{/Z}})_y$ embeds into the $I$-adic completion of $A_\mathfrak{q}$, and likewise for $x$ and $\mathfrak{p}$ (in fact they are faithfully flat algebras). Let $\widehat{M}$ denote the $I$-adic completion for now, then we can form the diagram
    \[
    \begin{tikzcd}
        (\mathcal{O}_{X_{/Z}})_y\arrow{r} \arrow[hookrightarrow]{d} & (\mathcal{O}_{X_{/Z}})_x\arrow[hookrightarrow]{d}\\
        \widehat{A_\mathfrak{q}} \arrow{r} & \widehat{A_\mathfrak{p}}
    \end{tikzcd}
    \]
    and so it would be sufficient to show that the bottom map is injective.
    
    As $\left(\widehat{A_\mathfrak{q}}\right)_\mathfrak{p}/\left(I^n\widehat{A_\mathfrak{q}}\right)_\mathfrak{p} =  \left(\widehat{A_\mathfrak{q}}/I^n\widehat{A_\mathfrak{q}}\right)_\mathfrak{p}$ and $\widehat{A_\mathfrak{q}}/I^n\widehat{A_\mathfrak{q}} = A_\mathfrak{q}/I^nA_\mathfrak{q}$ as we are completing with respect to $I$, we see that
    \[\widehat{\left(\widehat{A_\mathfrak{q}}\right)_\mathfrak{p}} = \varprojlim\left((\widehat{A_\mathfrak{q}}/I^n\widehat{A_\mathfrak{q}})_\mathfrak{p}\right) = \varprojlim \left((A_\mathfrak{q}/I^nA_\mathfrak{q})_\mathfrak{p}\right) = \varprojlim A_\mathfrak{p}/I^nA_\mathfrak{p} = \widehat{A_\mathfrak{p}}.\]
    $\widehat{A_\mathfrak{q}}$ is an integral domain by Lemma \ref{integralcomplete}, and thus $\widehat{A}_\mathfrak{q}$ injects into the first ring, and so we see the bottom map is an injection.
\end{proof}

\begin{prop}\label{StalksInjective}
    Let $X$ be an excellent normal scheme and $Z$ a closed subscheme. For any connected subset $U\subseteq X_{/Z}$ and $x\in U$, the map $H^0(U,\mathcal{O}_{X_{/Z}})\to (\mathcal{O}_{X_{/Z}})_x$ is injective.
    \begin{proof}
        Given a specialization $x\rightsquigarrow y$, the map $(\mathcal{O}_{X_{/Z}})_y\to (\mathcal{O}_{X_{/Z}})_x$ is injective by Lemma \ref{formalgenerization}. Suppose that $f\in H^0(U,\mathcal{O}_{X_{/Z}})$ is such that $f_x = 0$. Given any two points $x,y\in U$, as $U$ is connected, there exists a chain of irreducible components $Z_1,\dots, Z_n$ such that $x\in Z_1,\ Z_i\cap Z_{i+1}\neq\varnothing,\ y\in Z_n$. Using injectivity of the generization maps on stalks, we are able to pass $f_x = 0$ between the generic points of the $Z_i$ and their points of intersection to conclude that $f_y = 0$, and thus $f = 0$ as desired.
    \end{proof}
\end{prop}

Now letting $X = \Spec A$ and $Z = V(I)$, from \cite[Theorem III 5.1]{peskineDimensionProjectiveFinie1973} we see that $\mathrm{cd}(A,I)\leq \dim A - 2$ is equivalent to having the canonical map $H^0(X_{/Z},\mathcal{O}_{X_{/Z}}) = A\to H^0(X_{/Z}
\setminus\{\mathfrak{m}\},\mathcal{O}_{X_{/Z}})$ be an isomorphism. It is an injection as the kernel is simply functions which vanish at all points but at the closed point, but any formal function which vanishes at any point on a connected set is globally zero by \ref{StalksInjective}.

What we see then is that SVT($A,I$) is equivalent to asking that formal functions on the punctured spec of $X$ along $Z$ lift to formal functions along all of $Z$.

\section{Reduction to a Surface}\label{reduction-to-surface}
Let $A$ be a complete regular local ring with separably closed residue field, $I\trianglelefteq A$ an ideal with $\dim A/I\geq 2$ and $V(I)\setminus\{\mathfrak{m}\}$ connected. Set $X = \Spec A,\ U = X\setminus\{\mathfrak{m}\},\ Z= V(I)$. Let $\mathfrak{p}\supseteq I$ a prime ideal such that $\dim A/\mathfrak{p} = 2$, setting $Z' = V(\mathfrak{p})$, trivially we have that $U\cap Z'$ is connected. 

\begin{lem}\label{surfacelemma}
    Suppose that $\mathrm{SVT}(A,\mathfrak{p})$ holds, then $\mathrm{SVT}(A,I)$ holds as well.
\end{lem}
\begin{proof}
    The canonical map $A\to \Gamma(\mathfrak{U}',\mathcal{O}_{X_{/Z'}})$ is an isomorphism where $\mathfrak{U}'$ is the formal open subscheme $Z'\setminus\{\mathfrak{m}\}$. We have a map of schemes $Z'\to Z$ inducing a map $\Gamma(\mathfrak{U},\mathcal{O}_{X_{/Z}})\to \Gamma(\mathfrak{U}',\mathcal{O}_{X_{/Z'}})$ where $\mathfrak{U}$ is the formal open subscheme $Z\setminus\{\mathfrak{m}\}$, and this map factors the map $A\to \Gamma(\mathfrak{U}',\mathcal{O}_{X_{/Z'}})$ via the canonical map $A\to \Gamma(\mathfrak{U},\mathcal{O}_{X_{/Z}})$. Thus, to show that the latter map is surjective, it would suffice to show that $\Gamma(\mathfrak{U},\mathcal{O}_{X_{/Z}})\to \Gamma(\mathfrak{U}',\mathcal{O}_{X_{/Z'}})$ is injective.

By Proposition \ref{StalksInjective} the map from $\Gamma(\mathfrak{U},\mathcal{O}_{X_{/Z}})\to (\mathcal{O}_{X_{/Z}})_x$ is injective for all $x$, and similarly for $\mathfrak{U}'$ and $\mathcal{O}_X{_{/Z'}}$. We also note that the map $(\mathcal{O}_{X_{/Z}})_x\to (\mathcal{O}_{X_{/Z'}})_x$ induces an isomorphism on completions whenever this expression makes sense, meaning we can form the following commutative diagram, picking an arbitrary $x\in \mathfrak{U}'$:
\[
\begin{tikzcd}
    \Gamma(\mathfrak{U},\mathcal{O}_{X_{/Z}})\arrow{r} \arrow[hookrightarrow]{d} & \Gamma(\mathfrak{U'},\mathcal{O}_{X_{/Z'}})\arrow[hookrightarrow]{d}\\
    (\mathcal{O}_{X_{/Z}})_x\arrow{r} \arrow[hookrightarrow]{d} & (\mathcal{O}_{X_{/Z'}})_x\arrow[hookrightarrow]{d}\\
    (\mathcal{O}_{X_{/Z}})^{\hat{}}_x \arrow{r}{\cong} &(\mathcal{O}_{X_{/Z'}})^{\hat{}}_x 
\end{tikzcd}
\]
This forces the horizontal arrows to all be injective, which shows that the canonical map $A\to \Gamma(\mathfrak{U},\mathcal{O}_{X_{/Z}})$ is an isomorphism, so that SVT$(A,I)$ also holds.
\end{proof}

(c.f \cite[Theorem 3.2]{zhangSecondVanishingTheorem2025} for a more algebraic proof of this fact.)

\section{Reduction to an Algebraically Closed Residue Field}\label{algclosed}
Suppose that $(A,\mathfrak{m},k)$ is a complete Noetherian local ring with $k$ separably closed. Let $I$ be a prime ideal such that $V(I)$ is dimension 2. Let $B$ be the integral closure of $A/I$, then note that as $A/I$ is still complete Noetherian local it is henselian, and thus $B$ is a finite product of local rings, but as it is an integral domain this implies that $B$ is local. $A/I\to B$ is a finite extension and so it is surjective on spectra, and $B$ has depth 2 as it is normal and dimension 2. Let $A'$ be a gonflement \cite[\href{https://stacks.math.columbia.edu/tag/03C3}{Tag 03C3}]{stacks-project} of $A$ with residue field $\overline{k}$, by this we mean that $(A,\mathfrak{m},k)\to (A',\mathfrak{n},\overline{k})$ is a faithfully flat map of local rings, where $\mathfrak{n} = \mathfrak{m}A'$. We can and will assume that $A'$ is Noetherian and complete. Indeed, replace $A'$ with $\varprojlim A'/\mathfrak{m}^nA'$, then \cite[\href{https://stacks.math.columbia.edu/tag/05GH}{Tag 05GH}]{stacks-project} implies it is Noetherian while the local criterion for flatness \cite[\href{https://stacks.math.columbia.edu/tag/0523}{Tag 0523}]{stacks-project} implies it is still flat. As the associated graded of $A'$ is obtained by tensoring the associated graded of $A$ with $\overline{k}$ we see that $A'$ is a regular local ring.

Consider the diagram
\[\begin{tikzcd}
    A'\arrow{r} & A'/IA'\arrow{r}{\text{finite}} & B \otimes_A A'\\
    A\arrow{u}{\text{faithfully flat}} \arrow{r} &  A/I\arrow{u}{\text{faithfully flat}} \arrow{r}{\text{finite}} & B\arrow{u}{\text{faithfully flat}}
\end{tikzcd}\]
As $A'/IA'$ is a complete Noetherian local ring we see that $B\otimes_A A'$ is a finite product of local rings, we wish to see that it is local. To do this, note that every maximal ideal of $B\otimes_A A'$ lies over $\mathfrak{m}$, and so we may mod out the entire diagram by $\mathfrak{m}$ to get the following:
\[\begin{tikzcd}
    \overline{k}\arrow{r} & B/\mathfrak{m}B \otimes_k \overline{k}\\
    k\arrow{u} \arrow{r}{\text{finite}} & B/\mathfrak{m}B\arrow{u}
\end{tikzcd}\]
Now as $B/\mathfrak{m}B$ is a finite type algebra over $k$, it is geometrically connected if and only if its base change to $k^{\mathrm{sep}}$ is connected, but as $k$ is separably closed this just reduces to asking if $B/\mathfrak{m}B$ is connected, which it is as it is just a single point. Thus $B/\mathfrak{m}B\otimes_k \overline{k}$ is connected, implying that it is also local as it is an Artinian ring. This shows that $B\otimes_A A'$ is local since all its maximal ideals lie over $\mathfrak{m}$, and its reduction mod $\mathfrak{m}$ is local. As it is a flat extension of $B$ it has depth $\geq 2$, and so by Hartshorne's connectedness lemma we have that $\Spec (B\otimes_A A')\setminus\{\eta\}$ is connected where $\eta$ is the closed point, but as $\Spec B\otimes_A A'\to \Spec A'/IA'$ is surjective and closed, $\Spec (B\otimes_A A')\setminus\{\eta\}$ will surject onto $\Spec (A'/IA')\setminus\{\mathfrak{n}\}$, and since the former is connected, so too is the latter.

Now if we wish to show that $H^k_I(A) = 0$ under the assumption that $\dim V(I) \geq 2$ and that $\Spec (A/I)\setminus\{\mathfrak{m}\}$ is connected, we may base change by the faithfully flat extension $A'$ to ask when $H^k_{IA'}(A') = 0$, and the conditions about $V(I)$ being 2-dimensional or more and $V(I)$ minus the closed point being connected will be preserved, so we may assume that $k$ is algebraically closed.

If in any case $IA'$ is no longer such that $A'/IA'$ is a 2-dimensional integral domain, then we can pick $\mathfrak{p}\supseteq IA'$ such that $A'/\mathfrak{p}$ is a 2-dimensional integral domain. Then Lemma \ref{surfacelemma} tells us that we need only show SVT$(A',\mathfrak{p})$.

\begin{remark}\label{MaComment}
    It was pointed out to us by Linquan Ma that the reduction to a normal surface is not necessary for the argument in \secref{resolution}. For details on this refer to Remark \ref{MaProof}.

    We have chosen to keep these sections as we believe the methods may be interesting in their own right.
\end{remark}

\section{Reduction to a Normal Surface}\label{normalsurface}
Suppose that SVT$(S,J)$ holds where $S$ is a complete regular local ring, and $S/J$ is a normal, dimension 2 ring. Let $(A,\mathfrak{m},k)$ be a complete regular local ring, with $k$ algebraically closed. Let $\mathfrak{p}$ be a prime ideal with $\dim A/\mathfrak{p} = 2$, and set $X = \Spec A,\ Z = V(\mathfrak{p}),\ U = X\setminus\{\mathfrak{m}\}$. Consider the ring $C$ which is the normalization of $A/\mathfrak{p}$, then this is a finite extension of $A/\mathfrak{p}$, and will thus be complete and local. Let $x_1,\dots,x_r$ be a generating set for $C$'s maximal ideal, then consider the map $A\llbracket T_1,\dots, T_r\rrbracket\to C$ defined by sending $T_i$ to $x_i$. By Nakayama's lemma this map is surjective, and so $Z' = \Spec C$ is a closed subscheme of $X' = \Spec A\llbracket T_1,\dots, T_r\rrbracket$ lying over $Z$ under the canonical map $X'\to X$. Let $\eta$ be the maximal ideal of $A\llbracket T_1,\dots, T_r\rrbracket$.

\begin{prop}\label{normalization}
    Let $U' = X'\setminus \{\eta\}$ and $U = X\setminus \{\mathfrak{m}\}$. Let $\mathfrak{U}' = U'\cap X_{/Z'}$ and $\mathfrak{U} = U\cap X_{/Z}$, then there is a canonical injection
    \[A\subseteq \Gamma(\mathfrak{U},\mathcal{O}_{X_{/Z}})\hookrightarrow \Gamma(\mathfrak{U}',\mathcal{O}_{X'_{/Z'}}) = A\llbracket T_1,\dots,T_r\rrbracket,\]
    such that $\Gamma(\mathfrak{U},\mathcal{O}_{X_{/Z}})$ is fixed by every continuous $A$-algebra automorphism of $A\llbracket T_1,\dots, T_r\rrbracket$ compatible with the surjection to $C$.
\end{prop} 
\begin{proof}
    We construct the map as follows. The formal scheme $(\mathfrak{U}',\mathcal{O}_{X_{/Z'}}\rest{\mathfrak{U}'})$ is the completion of the scheme $U'$ along $Z'\cap U'$, but as $U'' = X'\setminus f^{-1}(\mathfrak{m})$ (where $f\colon X'\to X$) is an open subset of $U'$ also containing $Z'\cap U'$, it is also the completion of $U''$ along $Z'\cap U' = Z'\cap U''$. Then, we have a map $U''\to U$ which sends $Z'\cap U''$ into $Z\cap U$, inducing the desired map of formal schemes leading to the map described above.

    To see that the middle map is injective we can create the following diagram,
    \[
    \begin{tikzcd}
        \Gamma(\mathfrak{U},\mathcal{O}_{X_{/Z}}) \arrow{r} \arrow[hookrightarrow]{d} & \Gamma(\mathfrak{U}',\mathcal{O}_{X'_{/Z'}}) \arrow[hookrightarrow]{d}\\
        (\mathcal{O}_{X_{/Z}})_{f(x)}\arrow{r} \arrow[hookrightarrow]{d} & (\mathcal{O}_{X'_{/Z'}})_x \arrow[hookrightarrow]{d}\\
        (\mathcal{O}_{X_{/Z}})_{f(x)}^{\widehat{}} = \widehat{A}_{{\mathfrak{p}_x}\cap A} \arrow[hookrightarrow]{r} & (\mathcal{O}_{X'_{/Z'}})_{x}^{\widehat{}} = A\llbracket T_1,\dots, T_r\rrbracket_{\mathfrak{p}_x}^{\widehat{}}
    \end{tikzcd}
    \]
    which forces the horizontal arrows to be injective as well.

    We now show that $\Gamma(\mathfrak{U},\mathcal{O}_{X_{/Z}})$ has the property that it is fixed by every continuous $A$-algebra automorphism of $\Gamma(\mathfrak{U}',\mathcal{O}_{X'_{/Z'}})$ which is compatible with the projection to $C$. Indeed, such an automorphism will give rise to a diagram of the following form:
    
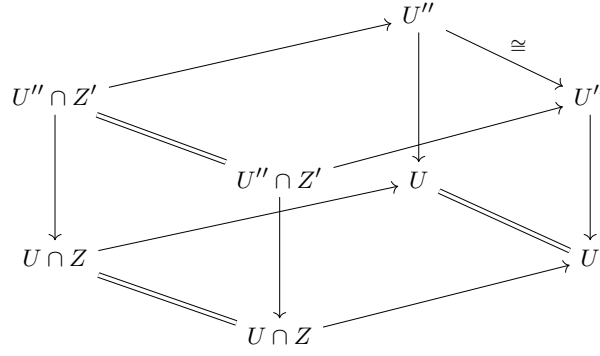
\[
    \begin{tikzcd}
    &&& {U''} && \\
    {U''\cap Z'} &&&&& {U''} \\
    && {U''\cap Z'} & U \\
    {U \cap Z} &&&&& U \\
    && {U \cap Z}
    \arrow["\cong", from=1-4, to=2-6]
    \arrow[from=1-4, to=3-4]
    \arrow[from=2-1, to=1-4]
    \arrow[equals, from=2-1, to=3-3]
    \arrow[from=2-1, to=4-1]
    \arrow[from=2-6, to=4-6]
    \arrow[from=3-3, to=2-6]
    \arrow[from=3-3, to=5-3]
    \arrow[equals, from=3-4, to=4-6]
    \arrow[from=4-1, to=3-4]
    \arrow[equals, from=4-1, to=5-3]
    \arrow[from=5-3, to=4-6]
\end{tikzcd}
    \]
    To ensure that the map $X'\to X'$ descends to $U''\to U''$, one merely uses continuity to ensure that the fiber over $\mathfrak{m}$ is sent to itself. Such a diagram induces the identity map on $(\mathfrak{U},\mathcal{O}_{X_{/Z}})$.
\end{proof}

We then simply have to show the following,
\begin{lem}
    With notation as above, given an element $f$ of $A\llbracket T_1,\dots,T_r\rrbracket$ fixed by every continuous $A$-algebra automorphism of $A\llbracket T_1,\dots,T_r\rrbracket$ compatible with the surjection to $C$, $f$ lies in $A$.
\end{lem}
\begin{proof}
    Let $g$ be an element of $\mathfrak{p}$, then note that $g$ is also in the kernel of the projection to $C$. For a fixed $1\leq i\leq r$, consider $f$ as a power series in $T_i$ over the ring $A\llbracket T_1,\dots,\hat{T_i},\dots,T_r\rrbracket$. The map defined by sending $T_i$ to $T_i + g$ is a continuous $A$-algebra automorphism, with inverse $T_i\mapsto T_i - g$, compatible with the projection to $C$ which is even a $A\llbracket T_1,\dots,\hat{T_i},\dots,T_r\rrbracket$-algebra automorphism. We see that as a power series,
    \[f(T_i + g) = f(T_i) + gf'(T_i) + (g^2),\]
    and so if $f$ is fixed by this automorphism we see that $f'(T_i)$ is divisible by g. The same holds true when replacing $g$ by $g^n$, and so we conclude that $f'(T_i)$ is divisible by arbitrarily high powers of $g$, and so it must be zero. In characteristic 0 this implies that $f$ does not depend on $T_i$, and so $f$ must be a constant. In characteristic $p$, this implies that $f(T_i)$ is a power series in $T_i^p$, and so there exists a $h(T_i)$ where $h(T_i^p) = f(T_i)$. Now using the Taylor expansion for $h$ expanded about $T_i^p$ yields the equation
    \[h(T_i^p) = f(T_i) = f(T_i + g) = h(T_i^p + g^p) = h(T_i^p) + h'(T_i^p)g^p + (g^{2p})\]
    From this we conclude that multiplication by $g^p$ sends $h'(T_i^p)$ into $(g^{2p})$, and again this holds for all powers of $g$ so $h'(T_i^p) = 0$. This implies that $h'(T_i) = 0$ and again $h$ is a power series in $T_i^p$. Repeating this, we see $f$ is a power series in $T_i^{p^k}$ for arbitrarily high $k$, which is impossible unless it is a constant in $T_i$. As this is true for all $i$, we see that $f$ is a constant and thus lies in $A$ as desired.
\end{proof}
Going forth, let us now assume that $A/I$ is normal. It is also worth noting that until now nothing was dependent on the characteristic.

\section{A Resolution}\label{resolution}
Let $Z'\to Z$ be a resolution of $Z$ factoring through the blowup at $\mathfrak{m}$ (This always exists for excellent surfaces, c.f \cite[\href{https://stacks.math.columbia.edu/tag/0ADW}{Chapter 0ADW}]{stacks-project}), and $P = \BP^N_A$ a projective space into which $Z'$ embeds. Let $J$ be the ideal sheaf of $Z'$ in $P$, $s$ be the closed point of $X$, and $Z'_s$ denote the fiber of $Z'$ over $s$. By construction $Z'_s$ is an effective Cartier divisor.

We will proceed in two steps, first proving that $H^0(Z'\setminus Z'_s,\mathcal{O}_{P_{/Z'}}) = H^0(Z',\mathcal{O}_{P_{/Z'}})$, and then showing that $H^0(Z',\mathcal{O}_{P_{/Z'}}) = H^0(P,\mathcal{O}_P)$. If we know this, then note that we have a commutative diagram
\[
\begin{tikzcd}
    \Gamma(Z'\setminus Z'_s, \mathcal{O}_{P_{/Z}}) & \Gamma(Z'\setminus Z'_s, \mathcal{O}_{P_{/Z}}) \arrow[swap]{l}{\cong} & \Gamma(P,\mathcal{O}_P) = A \arrow[swap]{l}{\cong}\\
    \Gamma(Z\setminus\{\mathfrak{m}\},\mathcal{O}_{X_{/Z}})\arrow[hookrightarrow]{u} & \Gamma(Z,\mathcal{O}_{X_{/Z}})\arrow{l} & \Gamma(X,\mathcal{O}_X) = A \arrow[swap]{l}{\cong} \arrow{u}{\cong}
\end{tikzcd}
\]
which would force the restriction map $\Gamma(Z,\mathcal{O}_{X_{/Z}})\to \Gamma(Z\setminus\{\mathfrak{m}\},\mathcal{O}_{X_{/Z}})$ to be surjective. The leftmost vertical arrow is injective by the method in the proof of Proposition \ref{normalization}; by looking at stalks and then completions we will be looking at some map $\widehat{A}_{\mathfrak{p}\cap A}\to (A[x_0/x_j,\dots, x_n/x_j])_\mathfrak{p}^{\widehat{}}$ which is injective.

Now, Gabber's result \cite[6.1.5]{berthelotAlterationsVarietesAlgebriques1997} implies, because we are in ramified mixed characteristic and $Z'$ is a regular subscheme of the regular scheme $P$, that the vector bundle $\shom_{\mathcal{O}_P}(J/J^2, \mathcal{O}_{Z'_s})$ is globally generated.

Dually, we have an injection
\[J/J^2\rest{Z'_s}\hookrightarrow \bigoplus \mathcal{O}_{Z'_s}\]
which is locally split. Taking symmetric powers we obtain injections
\[J^r/J^{r+1}\rest{Z'_s}\hookrightarrow \bigoplus \mathcal{O}_{Z'_s}\]
which are again locally split.

\begin{prop}\label{lifting}
    If $H^0(Z'_s,\mathcal{O}_{Z'_s}(cZ'_s))= 0$ for all $c > 0$, then $H^0(Z'\setminus Z'_s,\mathcal{O}_{P_{/Z'}}) = H^0(Z',\mathcal{O}_{P_{/Z'}})$.
\end{prop}
\begin{proof}
As $Z'$ is a regular scheme embedded in the regular scheme $P$, $J$ is locally generated by a regular sequence so that $J/J^2$ is locally free. This implies that for all $k$, $J^k/J^{k+1}$ are locally free, so any local section of $\mathcal{O}_P$ which is not a zero-divisor on $\mathcal{O}_P/J = \mathcal{O}_Z$ remains a non-zero divisor on $J^k/J^{k+1}$ for all $k$. This in turn implies that these are non zero-divisors on $\mathcal{O}_P/J^n$ for all $n$ by using the filtration by $J^k$. This means that effective Cartier divisors $D$ on $Z$, such as $Z'_s$, lift to effective Cartier divisors on $\mathcal{O}_P/J^n$ which by abuse of notation we will also denote $D$.

As $Z'_s$ is an effective Cartier divisor on $\mathcal{O}_P/J^n$ for each $n$, we have the following relationship by \cite[Lemma 4]{goodmanSchemesFinitedimensionalCohomology1969}:
\[\Gamma(Z'\setminus Z'_s,\mathcal{O}_P/J^n) = \varinjlim_c \Gamma(Z',(\mathcal{O}_P/J^n)(cZ'_s)).\]
We also have the following short exact sequence,
\[0\to (\mathcal{O}_P/J^n)(cZ'_s)\to (\mathcal{O}_P/J^n)((c+1)Z'_s)\to (\mathcal{O}_P/J^n)\rest{Z'_s}((c+1)Z'_s)\to 0,\]
and so if we can show that $H^0(Z'_s,(\mathcal{O}_P/J^n)\rest{Z'_s}(cZ'_s)) = 0$ for all $c > 0$ all the groups $\Gamma(Z',(\mathcal{O}_P/J^n)(cZ'_s))$ are equal to $\Gamma(Z',\mathcal{O}_p/J^n)$, and all functions lift over $Z'_s$. As this holds for each $n$, it will hold in the limit.

Note that $(\mathcal{O}_P/J^n)\rest{Z'_s}(cZ'_s)$ has a filtration by things of the form $(J^r/J^{r+1})\rest{Z'_s}(cZ'_s)$, and this sheaf is a subsheaf of a direct sum of sheaves of the form $\mathcal{O}_{Z'_s}(cZ'_s)$, so in particular it suffices to show that $H^0(Z'_s,\mathcal{O}_{Z'_s}(cZ'_s)) = 0$ for $c > 0$.
\end{proof}

Now the Hodge Index Theorem \cite[Lemma 14.1]{lipmanRationalSingularitiesApplications1969}, which Lipman attributes to Du Val, implies that the intersection matrix of the components of $(Z'_s)_{\mathrm{red}}$ is negative definite, which implies that the self intersection of any divisor set-theoretically supported on $Z'_s$ is negative. In addition, note that $Z'_s$ is antinef due to its dual $\mathfrak{m}\mathcal{O}_{Z'}$ being globally generated.

\begin{lem}
    If $Z$ is the spectrum of a normal, local 2-dimensional ring with closed point $s$ and $Z'\to Z$ is a resolution of $Z$, then $H^0(Z'_s,\mathcal{O}_{Z'_s}(cZ'_s)) = 0$ for all $c > 0$.
\end{lem}
\begin{proof}
Write $Z'_s = F_1 + \dots + F_n$, where each $F_i$ is an integral curve (allowing duplicates), and let $S$ be the multiset $\{F_i\}$. For ease of writing, set $F = Z'_s$.

By negative-definiteness, $F^2 < 0$, but if we expand it out we have $F^2 = \sum_1^n F_i\cdot F$, and so one $F_i\cdot F$ must be negative. Let $C_1$ be one such curve such that $C_1\cdot F < 0$, and set $D_1 = C_1$.

Note that because we have an integral curve $D_1 = C_1$ intersecting $F$ negatively, we have $H^0(\mathcal{O}_{D_1}(cF))$ = 0.

Assume that we have chosen $C_1,\dots,C_i$ such that for all $1\leq j\leq i - 1, C_{j+1} \in S\setminus\{C_1,\dots,C_j\}$, $D_j = \sum_1^j C_k$, and $H^0(\mathcal{O}_{D_j}(cF))$ = 0.

For $1\leq j\leq i$, relabel the $F_j$ so that $F_j = C_j$. Then, $(F_{i+1} + \dots + F_n)^2 < 0$, again by negative definiteness, but again we can write this intersection as $\sum_{i+1}^n F_j\cdot(F_{i+1} + \dots + F_n)$, so some $F_j\cdot(F_{i+1} + \dots + F_n) = F_j(F - D_i)$ is negative. Let $C_{i+1}$ be one such $F_j$, and then set $D_{i+1} = D_i + C_{i+1}$.

By \cite[\href{https://stacks.math.columbia.edu/tag/0C4T}{Tag 0C4T}]{stacks-project} we have the following short exact sequence,

\[0\to \mathcal{O}_{C_{i+1}}(-D_i)\to \mathcal{O}_{D_{i+1}}\to \mathcal{O}_{D_i}\to 0.\] 

By twisting by $cF$ we get that

\[0\to \mathcal{O}_{C_{i+1}}((c-1)F + F - D_i)\to \mathcal{O}_{D_{i+1}}(cF)\to \mathcal{O}_{D_i}(cF)\to 0\]  
is exact, and since $C_{i+1}$ intersects $F - D_i$ negatively and $(c-1)F$ non-positively as $F$ is anti-nef the first term has no global sections. When combined with our inductive hypothesis we have $H^0(\mathcal{O}_{D_{i+1}}(cF)) = 0$.

Letting $i = n$, since we have $D_n = F$, we conclude that $H^0(\mathcal{O}_{F}(cF)) = 0$, but $F$ was just $Z'_s$.
\end{proof}

We thus can conclude that
\[\Gamma(Z'\setminus Z'_s,\mathcal{O}_{P_{/Z'}}) =\varprojlim \Gamma(Z'\setminus Z'_s, \mathcal{O}_P/J^n) = \varprojlim\Gamma(Z',\mathcal{O}_P/J^n) = \Gamma(Z',\mathcal{O}_{P_{/Z'}}).\]

Our next goal will be to show that any formal function of $P$ along $Z'$ actually comes from a global function on $P$. Denote $E = \BP^N_{\kappa(\mathfrak{m})}$, and $E_n = \BP^N_{A/\mathfrak{m}^n}$ the $n$-th infinitesimal neighborhood of $E$ inside of $P$. We first show that
\[\Gamma((E_n)_{/E_n\cap Z'},\mathcal{O}_{(E_n)_{/E_n\cap Z'}}) = A/\mathfrak{m}^n.\]

Now for each $n$, we obtain a functorial map $\varphi_n\colon \Gamma(E_n,\mathcal{F})\to\Gamma((E_n)_{/E_n\cap Z'},\mathcal{F}_{/E_n\cap Z'})$ for any quasicoherent module $\mathcal{F}$. Noting that $E_1\cap Z'\subseteq\dots E_n\cap Z'$ are infinitesimal thickenings of each other within $E_n$, we note that we may as well complete with respect to any other $E_k\cap Z'$ for $1\leq k < n$. Noting that for $\mathcal{F} = \mathcal{O}_{E_n}$ the domain of $\varphi_n$ is $A/\mathfrak{m}^n$, our goal is to show that $\varphi_n$ is an isomorphism for these choices of $\mathcal{F}$, which we will do inductively.

\begin{prop}\label{thick-sections}
    For $\mathcal{F} = \mathcal{O}_{E_n}$, the map $\varphi_n$ is an isomorphism. Therefore, $\Gamma((E_n)_{/E_n\cap Z'},\mathcal{O}_{(E_n)_{/E_n\cap Z'}}) = \Gamma(E_n,\mathcal{O}_{E_n}) = A/\mathfrak{m}^n$.
\end{prop}
\begin{proof}
For $n = 1$, we note that $E_1 = P^N_{\kappa(\mathfrak{m})}$ and $E_1\cap Z'$ is a connected curve as $Z$ is normal (Stein Factorization), and so \cite[Theorem V.3.1]{hartshorneAmpleSubvarietiesAlgebraic1970} implies that $\varphi_1$ is an isomorphism.

Proceeding by induction, we assume that $\varphi_i$ is an isomorphism for $i < n$. We can write the following exact sequence,
\[0\to \mathfrak{m}^{n-1}\mathcal{O}_{E_n}\to \mathcal{O}_{E_n}\to \mathcal{O}_{E_{n-1}}\to 0.\]
Observe that $\mathfrak{m}^{n-1}\mathcal{O}_{E_n}$ is supported on $E_1$, and when restricted to $E_1$ is isomorphic to a direct sum of copies of the structure sheaf, given by the degree $n-1$ monomials in a set of regular system of parameters of $\mathfrak{m}$.
Noting further that $\mathcal{O}_{E_{n-1}}$ is supported on $E_{n-1}$, by replacing the subscheme we are completing at by the corresponding smaller subscheme and then using Lemma \ref{restrict-complete}, we obtain the following map between exact sequences
\[
\begin{tikzcd}[column sep = 0.6cm]
    0\arrow{r}\arrow{d} & \Gamma(E_1,\bigoplus\mathcal{O}_{E_1}\arrow{r}\arrow{d}{\varphi_1}) & \Gamma(E_n,\mathcal{O}_{E_n})\arrow{r}\arrow{d}{\varphi_n} & \Gamma(E_{n-1},\mathcal{O}_{E_{n-1}})\arrow{r}\arrow{d}{\varphi_{n-1}}\arrow{r} & H^1(E_1,\bigoplus\mathcal{O}_{E_1})\arrow{d}\\
    0\arrow{r} & \Gamma(\widehat{E_1},\widehat{\bigoplus\mathcal{O}_{E_1}})\arrow{r} & \Gamma(\widehat{E_n},\widehat{\mathcal{O}_{E_n}})\arrow{r} & \Gamma(\widehat{E_{n-1}},\widehat{\mathcal{O}_{E_{n-1}}})\arrow{r} & ?
\end{tikzcd}
\]
where it does not matter what is in the bottom right corner. The first, second, and fourth maps are isomorphisms by induction, and the fifth map is injective as the cohomology group vanishes, and so by the five lemma $\varphi_n$ is an isomorphism.
\end{proof}

\begin{prop}
    With notation like before, $\Gamma(P_{/Z'},\mathcal{O}_{P_{/Z'}})=\Gamma(P,\mathcal{O}_P)$.
\end{prop}
\begin{proof}

Using Lemma \ref{sheaf thickenings} and Proposition \ref{thick-sections}, we see that
\[\Gamma(P_{/E\cap Z'},\mathcal{O}_{E_{/E\cap Z'}}) = \varprojlim \Gamma((E_n)_{/E_n\cap Z'}, \mathcal{O}_{(E_n)_{/E_n\cap Z'}}) = \varprojlim A/\mathfrak{m}^n = A = \Gamma(P,\mathcal{O}_P).\]

Thus, given $s\in \Gamma(P_{/Z'},\mathcal{O}_{P_{/Z'}})$, there exists a function $s'\in \Gamma(P,\mathcal{O}_P)$ whose image agrees with $s$ at all points in $Z'\cap E$, but this is enough to say it agrees at all points. Indeed, like in the proof of \secref{reduction-to-surface}, the map $\Gamma(P_{/Z'},\mathcal{O}_{P_{/Z'}})\to \Gamma(P_{/Z'\cap E},\mathcal{O}_{/Z'\cap E})$ is injective, so it is enough to check that they agree after restricting.
\end{proof}

With this, the SVT has been proven for all pairs $(A,I)$ where $A$ is a ramified regular local ring, and as a result we obtain the following:
\maintheorem*

\begin{remark}\label{equichar}
    In equicharacteristic the normal bundle $\shom_{\mathcal{O}_P}(J/J^2,\mathcal{O}_{Z'_s})$ will be globally generated since the tangent bundle of $\mathcal{O}_P$ is globally generated by the Euler sequence and then we may apply the conormal sequence (note that we are able to work over an algebraically closed field so smoothness and regularity are the same). This means that the proof just presented works directly in equicharacteristic. This was pointed out to us by Linquan Ma.
\end{remark}

\begin{remark}\label{MaProof}
    Like was promised in Remark \ref{MaComment}, we present the necessary details to carry out the proof of this section without reducing to a normal surface. This proof was communicated to us by Linquan Ma.

    Letting $Z'$ be a resolution of $Z$, as $Z'$ is normal it will factor through the normalization $Z''$. The normalization of $A/I$ is local as $A/I$ is henselian, and the normalization is a domain, and the inverse image of $s\in Z$ and $s''\in Z''$ are set-theoretically equal, and thus define the same reduced divisor. This is enough to ensure that the intersection matrix obtained from $Z'_s$ is negative-definite as the one obtained from $Z'_{s''}$ is, and that $Z'_s$ is connected.
    
    The cited result \cite[Theorem V.3.1]{hartshorneAmpleSubvarietiesAlgebraic1970} is done for algebraically closed fields, which is why we still needed to make that reduction. Although one might be tempted to use \cite[Theorem 3.3]{hironakaFormalFunctionsFormal1968} which only requires the field to be infinite, this result will say that $Z'_s$ is universally $G_1$, which doesn't imply $G_1$ which is the condition that the map $H^0(E_1,\mathcal{O}_{E_1})\to H^0(E_1\cap Z'_s,\mathcal{O}_{(E_1)_{/E_1\cap Z'}})$ is an isomorphism. And indeed, as we have to factor through the normalization, if the normalization introduces a non-trivial residue extension, then $Z'_{s''}$ is ascheme over $s''$. Then $(E_1)_{/E_1\cap Z'}$ is a formal scheme over $s''$ and thus has global sections larger than $k$.

    It is, however, worth noting that for the proof of Theorem \ref{unramified} we do actually require the reduction to $A/I$ being normal.
\end{remark}

\section{A Unified Proof}
In \cite{hunekeVanishingLocalCohomology1990} a unified proof of the equicharacteristic versions of the SVT was given. As our proof utilizes the ramification in an essential way, the method cannot directly be applied in unramified mixed characteritic, although as noted in Remark \ref{equichar} it can work in equicharacteristic. However, we can show that the SVT holding for ramified rings implies the SVT for unramified regular local rings.

\begin{thm}\label{unramified}
    If SVT holds for ramified regular local rings, then it holds for unramified regular local rings.
\end{thm}
\begin{proof}
    Let $A$ be an unramified regular local ring whose residue field is characteristic $p > 0$, and $I$ an ideal such that $\dim A/I\geq 2$ and the punctured spec of $A^{\dagger}/IA^{\dagger}$ is connected. By \secref{reduction-to-surface}, \secref{algclosed}, and \secref{normalsurface} which did not use the ramified hypothesis, we may assume that $A$ is complete, has an algebraically closed residue field, and that $A/I$ is a normal surface. Note that this implies $A/I$ has depth 2.

    Consider the ring $B = \widehat{\left(\frac{A[x]}{x^2-p}\right)}_{(x,\mathfrak{m})}$, then $B$ is a faithfully flat ramified complete regular local $A$-algebra with an algebraically closed residue field. Then $B/IB$ is a faithfully flat local $A/I$ algebra, and thus has depth at least 2, and thus its punctured spec is connected. Theorem \ref{maintheorem} then implies that $H^{n-1}_{IB}(B) = H^{n-1}_I(A)\otimes B = 0$, and so $H^{n-1}_I(A) = 0$ and SVT$(A,I)$ holds.
\end{proof}

\begin{remark}
    Combined with Remark \ref{equichar} we see that the method of this paper can prove the SVT for all characteristics, yielding a unified proof.

    We were also able to adapt the method of proof of this paper to positive characteristic without assuming the global generation of the normal bundle by utilizing the result in ramified mixed characteristic which we demonstrate below, although this is unnecessary by the above consideration.
\end{remark}

\begin{thm}
    If SVT holds for ramified regular local rings, then it holds for positive characteristic regular local rings.
\end{thm}
\begin{proof}
    Let $A$ be a characteristic $p > 0$ regular local ring and $I$ an ideal such that $\dim A/I\geq 2$ and the punctured spec of $A^{\dagger}/IA^{\dagger}$ is connected. By \secref{reduction-to-surface}, \secref{algclosed}, and \secref{normalsurface} which did not use the ramified hypothesis, we may assume that $A$ is complete, has an algebraically closed residue field, and that $A/I$ is a normal surface.

    As $A$ is complete it is of the form $k\llbracket X_1,\dots, X_n\rrbracket$. Let $R = \widehat{\left(\frac{W(k)[x]}{x^2-p}\right)}_{(x,\mathfrak{m})}$, and then we have a surjection $R\to A$, let $I$ denote the kernel of this surjection. Note that $R$ is a ramified complete regular local ring with algebraically closed residue field. Let $\mathfrak{n}$ denote the maximal ideal of $R$.

    We will proceed as in \secref{resolution} letting $Z = \Spec A/I$, $Z'\to Z$ a resolution factoring through the blowup of $Z$ at the closed point $s$, and $P = \BP^N_A$ a projective space which $Z'$ embeds into. Let $J$ be the ideal sheaf of $Z'$ in $P$. Observe that in the proof of Proposition \ref{lifting} it was noted that it sufficed to show $H^0(Z'_s,(\mathcal{O}_P/J^n)\rest{Z'_s}(cZ'_s)) = 0$ for all $c > 0$. Let $P' = \BP^N_R$, and note that $P$ embeds into $P'$. Let the ideal sheaf of $Z'$ in $P'$ be denoted by $J'$.

    Note that $\mathcal{O}_P = \mathcal{O}_{P'}/I\mathcal{O}_{P'}$. We have that $J = (J' + I\mathcal{O}_{P'})/I\mathcal{O}_{P'}$ from which it follows that $J^n = ((J')^n + I\mathcal{O}_{P'})/I\mathcal{O}_{P'}$. Now definitionally:
    \[\mathcal{O}_{P'}/(J')^n\rest{Z'_s} = \frac{\mathcal{O}_{P'}}{(J')^n + \mathfrak{n}\mathcal{O}_{P'}}\]
    while
    \[\mathcal{O}_P/J^n\rest{Z'_s} = \frac{\mathcal{O}_{P}}{J^n + \mathfrak{m}\mathcal{O}_{P}}.\]
    Substituting in our other description for $J^n$ we see that
    \[\mathcal{O}_P/J^n\rest{Z'_s} = \frac{\mathcal{O}_{P'}/I\mathcal{O}_{P'}}{((J')^n + I\mathcal{O}_{P'} + \mathfrak{n}\mathcal{O}_{P'})/I\mathcal{O})_{P'}} = \frac{\mathcal{O}_{P'}}{((J')^n + I\mathcal{O}_{P'} + \mathfrak{n}\mathcal{O}_{P'}}.\]
    However $I\subseteq \mathfrak{n}$ and so this reduces to $\frac{\mathcal{O}_{P'}}{((J')^n + \mathfrak{n}\mathcal{O}_{P'}} = \mathcal{O}_{P'}/(J')^n\rest{Z'_s}$. Thus, $H^0(Z'_s,\mathcal{O}_{P_{/Z'}}) = H^0(Z'_s,\mathcal{O}_{P'_{/Z'}}) = 0$ by our proof in the ramified case.

    The rest of the proof of \secref{resolution} can then be carried out for $P$ and $Z'$ like normal to conclude that SVT$(A,I)$ holds.
\end{proof}

\section{Applications}
\bmhead{\large A Relation to Depth}
Suppose $A$ is a regular local ring, and $I$ an ideal such that $\mathrm{depth}(A/I)\geq c$. If $c = 2$, then we see that $A^{\dagger}/IA^{\dagger}$ also has depth at least 2, implying that its punctured spec is connected by Hartshorne's connectedness lemma. Clearly it is also at least 2-dimensional, and so by the SVT we see that $\mathrm{cd}(A,I)\leq \dim A - 2$. We can ask if for general $c$ we have $\mathrm{cd}(A,I)\leq \dim A - c$, and we will now make some observations.

In equicharacteristic $p$, by the results of Peskine-Szpiro \cite{peskineDimensionProjectiveFinie1973} we see that the result holds for all $c$.

For $c = 1$, Hartshorne-Lichtenbaum vanishing will prove the result. Indeed, $\widehat{A}$ is a regular local ring, and so the only minimal prime is $(0)$. Then, assuming that $\mathrm{depth}(A/I)\geq 1$ then clearly $\mathrm{depth}(\hat{A}/I\hat{A})\geq 1$ and it is also clearly at least 1-dimensional.

If $c = 3$, then it is known for rings essentially of finite type over fields \cite[Corollary 2.8]{daoRelationshipDepthCohomological2016}.

For $c = 4$ counterexamples are known in characteristic 0, like when $I$ is defined by the segre embedding of two projective spaces.
\bmhead{\large Connected Components}
\begin{thm}
    \emph{(}c.f \cite[Theorem 3.12]{hernandezCohomologicalDimensionLyubeznik2018}\emph{)} Let $A$ be an $n$-dimensional regular local ring, with separably closed residue field $k$. Let $I$ be an ideal of $A$ such that $\dim A/I\geq 2$. Let the punctured spectrum of $A/I$ have $t$ connected components, then $H^{n-1}_I(A) = E(k)^{\oplus t-1}$. Thus, $t = \dim_k\Hom(k,H^{n-1}_I(A)) + 1$.
\end{thm}
\begin{proof}
    We will fix ideals $I_1,\dots, I_t$ such that the complements of $V(I_i)\setminus\{\mathfrak{m}\}$ are the connected components of the punctured spec of $A/I$.
    
    We will induct on $t$. If $t = 1$ then by the SVT we see that $H^{n-1}_I(A) = 0$ as desired. Now assume that for $t > 1$ it is known that for any ideal $J$ such that the punctured spectrum of $A/J$ has $r < t$ connected components, we have $H^{n-1}_J(A) = E(k)^{\oplus r-1}$. In particular, this will hold for $J = I_1\cap\dots\cap I_{t-1}$. The Mayer-Vietoris sequence for local cohomolgy gives us an exact sequence
    \[H^{n-1}_{J + I_t}(A)\to H^{n-1}_J(A)\oplus H^{n-1}_{I_t}(A)\to H^{n-1}_I(A)\to H^n_{J+ I_t}(A)\to H^n_J(A)\oplus H^n_{I_t}(A).\]
    As none of $I_1 + \dots + \hat{I_k} + \dots + I_t$ are $\mathfrak{m}$-primary, this implies that $A/I_t$ and $A/J$ are both at least 1-dimensional. Thus, by Hartshorne-Lichtenbaum vanishing, the last term is 0. By SVT, $H^{n-1}_{I_t}(A) = 0$, and $J + I_t$ is $\mathfrak{m}$-primary, and since $A$ is a regular local ring this means $H^{n-1}_{J + I_t}(A) = 0$ and $H^n_{J + I_t}(A) = E(k)$. Thus, this exact sequences becomes
    \[0\to E(k)^{\oplus t-2}\to H^{n-1}_I(A)\to E(k)\to 0,\]
    and since $E(k)^{\oplus t-2}$ is injective this splits giving us that $H^{n-1}_I(A) = E(k)^{\oplus t-1}$.
\end{proof}
\begin{remark}
    The proof given above is essentially the same proof as in \cite[Theorem 3.12]{hernandezCohomologicalDimensionLyubeznik2018}. There, for the unramified case they relate this number to the mixed characteristic Lyubeznik numbers, which we do not choose to do so that we can make a statement ambivalent to characteristic. It is worth noting that Theorem 3.10 and 3.11 of their paper would appear to work in our ramified case as well, as they state that their Theorem 3.9, attributed to Faltings as \cite[Korollar 2]{faltingsUberLokaleKohomologiegruppen1980}, works for all regular local rings. However, an issue is brought up in \cite[Remark 3.8]{bataviaVanishingLocalCohomology2026}, where they state that in communications with the authors that Theorem 3.9 in the mixed characteristic case was based on a misinterpretation of Faltings work, and does not appear anywhere in the literature. Batavia essentially fixes this in the case that $A$ is unramified by proving the missing result \cite[Theorem 3.5]{bataviaVanishingLocalCohomology2026}, but the proof is complex and crucially uses that $A$ is unramified, both to write it down as a power series ring over a Cohen ring, and also to ensure that after modding out by $p$ you still have a regular local ring. To the author's knowledge, the crucial result \cite[Korollar 2]{faltingsUberLokaleKohomologiegruppen1980} on which Theorem 3.10 and 3.11 of \cite{hernandezCohomologicalDimensionLyubeznik2018} depends on is still unproven in the ramified case, and so we are unable to adapt these results to the ramified case.
\end{remark}


\bibliography{bibliography}% common bib file


\begin{thebibliography}{19}
\ifx \bisbn   \undefined \def \bisbn  #1{ISBN #1}\fi
\ifx \binits  \undefined \def \binits#1{#1}\fi
\ifx \bauthor  \undefined \def \bauthor#1{#1}\fi
\ifx \batitle  \undefined \def \batitle#1{#1}\fi
\ifx \bjtitle  \undefined \def \bjtitle#1{#1}\fi
\ifx \bvolume  \undefined \def \bvolume#1{\textbf{#1}}\fi
\ifx \byear  \undefined \def \byear#1{#1}\fi
\ifx \bissue  \undefined \def \bissue#1{#1}\fi
\ifx \bfpage  \undefined \def \bfpage#1{#1}\fi
\ifx \blpage  \undefined \def \blpage #1{#1}\fi
\ifx \burl  \undefined \def \burl#1{\textsf{#1}}\fi
\ifx \doiurl  \undefined \def \doiurl#1{\url{https://doi.org/#1}}\fi
\ifx \betal  \undefined \def \betal{\textit{et al.}}\fi
\ifx \binstitute  \undefined \def \binstitute#1{#1}\fi
\ifx \binstitutionaled  \undefined \def \binstitutionaled#1{#1}\fi
\ifx \bctitle  \undefined \def \bctitle#1{#1}\fi
\ifx \beditor  \undefined \def \beditor#1{#1}\fi
\ifx \bpublisher  \undefined \def \bpublisher#1{#1}\fi
\ifx \bbtitle  \undefined \def \bbtitle#1{#1}\fi
\ifx \bedition  \undefined \def \bedition#1{#1}\fi
\ifx \bseriesno  \undefined \def \bseriesno#1{#1}\fi
\ifx \blocation  \undefined \def \blocation#1{#1}\fi
\ifx \bsertitle  \undefined \def \bsertitle#1{#1}\fi
\ifx \bsnm \undefined \def \bsnm#1{#1}\fi
\ifx \bsuffix \undefined \def \bsuffix#1{#1}\fi
\ifx \bparticle \undefined \def \bparticle#1{#1}\fi
\ifx \barticle \undefined \def \barticle#1{#1}\fi
\bibcommenthead
\ifx \bconfdate \undefined \def \bconfdate #1{#1}\fi
\ifx \botherref \undefined \def \botherref #1{#1}\fi
\ifx \url \undefined \def \url#1{\textsf{#1}}\fi
\ifx \bchapter \undefined \def \bchapter#1{#1}\fi
\ifx \bbook \undefined \def \bbook#1{#1}\fi
\ifx \bcomment \undefined \def \bcomment#1{#1}\fi
\ifx \oauthor \undefined \def \oauthor#1{#1}\fi
\ifx \citeauthoryear \undefined \def \citeauthoryear#1{#1}\fi
\ifx \endbibitem  \undefined \def \endbibitem {}\fi
\ifx \bconflocation  \undefined \def \bconflocation#1{#1}\fi
\ifx \arxivurl  \undefined \def \arxivurl#1{\textsf{#1}}\fi
\csname PreBibitemsHook\endcsname

%%% 1
\bibitem[\protect\citeauthoryear{Lipman}{1969}]{lipmanRationalSingularitiesApplications1969}
\begin{botherref}
\oauthor{\bsnm{Lipman}, \binits{J.}}:
Rational singularities, with applications to algebraic surfaces and unique factorization.
Institut des Hautes \'Etudes Scientifiques. Publications Math\'ematiques
(36),
195--279
(1969)
\end{botherref}
\endbibitem

%%% 2
\bibitem[\protect\citeauthoryear{Hartshorne}{1968}]{hartshorneCohomologicalDimensionAlgebraic1968}
\begin{barticle}
\bauthor{\bsnm{Hartshorne}, \binits{R.}}:
\batitle{Cohomological dimension of algebraic varieties}.
\bjtitle{Annals of Mathematics. Second Series}
\bvolume{88},
\bfpage{403}--\blpage{450}
(\byear{1968})
\doiurl{10.2307/1970720}
\end{barticle}
\endbibitem

%%% 3
\bibitem[\protect\citeauthoryear{Ogus}{1973}]{ogusLocalCohomologicalDimension1973}
\begin{barticle}
\bauthor{\bsnm{Ogus}, \binits{A.}}:
\batitle{Local cohomological dimension of algebraic varieties}.
\bjtitle{Annals of Mathematics. Second Series}
\bvolume{98},
\bfpage{327}--\blpage{365}
(\byear{1973})
\doiurl{10.2307/1970785}
\end{barticle}
\endbibitem

%%% 4
\bibitem[\protect\citeauthoryear{Peskine and Szpiro}{1973}]{peskineDimensionProjectiveFinie1973}
\begin{botherref}
\oauthor{\bsnm{Peskine}, \binits{C.}},
\oauthor{\bsnm{Szpiro}, \binits{L.}}:
Dimension projective finie et cohomologie locale. {{Applications}} \`a la d\'emonstration de conjectures de {{M}}. {{Auslander}}, {{H}}. {{Bass}} et {{A}}. {{Grothendieck}}.
Institut des Hautes \'Etudes Scientifiques. Publications Math\'ematiques
(42),
47--119
(1973)
\end{botherref}
\endbibitem

%%% 5
\bibitem[\protect\citeauthoryear{Huneke and Lyubeznik}{1990}]{hunekeVanishingLocalCohomology1990}
\begin{barticle}
\bauthor{\bsnm{Huneke}, \binits{C.}},
\bauthor{\bsnm{Lyubeznik}, \binits{G.}}:
\batitle{On the vanishing of local cohomology modules}.
\bjtitle{Inventiones Mathematicae}
\bvolume{102}(\bissue{1}),
\bfpage{73}--\blpage{93}
(\byear{1990})
\doiurl{10.1007/BF01233420}
\end{barticle}
\endbibitem

%%% 6
\bibitem[\protect\citeauthoryear{Zhang}{2025}]{zhangSecondVanishingTheorem2025}
\begin{barticle}
\bauthor{\bsnm{Zhang}, \binits{W.}}:
\batitle{The second vanishing theorem for local cohomology modules}.
\bjtitle{Mathematical Research Letters}
\bvolume{32}(\bissue{6}),
\bfpage{2039}--\blpage{2062}
(\byear{2025})
\doiurl{10.4310/mrl.260123201714}
\end{barticle}
\endbibitem

%%% 7
\bibitem[\protect\citeauthoryear{Bhattacharyya}{2020}]{bhattacharyyaNoteSecondVanishing2020}
\begin{botherref}
\oauthor{\bsnm{Bhattacharyya}, \binits{R.}}:
A Note on the Second Vanishing Theorem.
arXiv
(2020).
\doiurl{10.48550/arXiv.2004.02075}
\end{botherref}
\endbibitem

%%% 8
\bibitem[\protect\citeauthoryear{Asgharzadeh et~al.}{2023}]{asgharzadehSurjectivityLocalCohomology2023}
\begin{barticle}
\bauthor{\bsnm{Asgharzadeh}, \binits{M.}},
\bauthor{\bsnm{Ishiro}, \binits{S.}},
\bauthor{\bsnm{Shimomoto}, \binits{K.}}:
\batitle{Surjectivity of some local cohomology map and the second vanishing theorem}.
\bjtitle{Proceedings of the American Mathematical Society}
\bvolume{151}(\bissue{7}),
\bfpage{2847}--\blpage{2862}
(\byear{2023})
\doiurl{10.1090/proc/16340}
{\href{https://arxiv.org/abs/2107.09041}{{arXiv:2107.09041}}}
{[math.AC]}
\end{barticle}
\endbibitem

%%% 9
\bibitem[\protect\citeauthoryear{Ma}{2026}]{maInfinitelyManyAssociated2026}
\begin{botherref}
\oauthor{\bsnm{Ma}, \binits{L.}}:
Infinitely Many Associated Primes of Local Cohomology Modules of Ramified Regular Local Rings.
arXiv
(2026).
\doiurl{10.48550/arXiv.2604.10964}
\end{botherref}
\endbibitem

%%% 10
\bibitem[\protect\citeauthoryear{Batavia}{2026}]{bataviaVanishingLocalCohomology2026}
\begin{botherref}
\oauthor{\bsnm{Batavia}, \binits{M.}}:
Vanishing of Local Cohomology in Unramified Mixed Characteristic.
arXiv
(2026).
\doiurl{10.48550/arXiv.2602.22191}
\end{botherref}
\endbibitem

%%% 11
\bibitem[\protect\citeauthoryear{Grothendieck}{1960}]{grothendieckElementsGeometrieAlgebrique1960}
\begin{botherref}
\oauthor{\bsnm{Grothendieck}, \binits{A.}}:
\'el\'ements de g\'eom\'etrie alg\'ebrique. {{I}}. {{Le}} langage des sch\'emas.
Institut des Hautes \'Etudes Scientifiques. Publications Math\'ematiques
(4),
228
(1960)
\end{botherref}
\endbibitem

%%% 12
\bibitem[\protect\citeauthoryear{Stacks Project~Authors}{2026}]{stacks-project}
\begin{botherref}
\oauthor{\bsnm{Stacks Project~Authors}, \binits{T.}}:
The {{Stacks Project}}.
https://stacks.math.columbia.edu
(2026)
\end{botherref}
\endbibitem

%%% 13
\bibitem[\protect\citeauthoryear{Berthelot}{1997}]{berthelotAlterationsVarietesAlgebriques1997}
\begin{bchapter}
\bauthor{\bsnm{Berthelot}, \binits{P.}}:
\bctitle{Alt\'erations de vari\'et\'es alg\'ebriques (d'apr\`es {{A}}. {{J}}. de {{Jong}})}.
In: \bbtitle{S\'eminaire {{Bourbaki}} : Volume 1995/96, Expos\'es 805-819}.
\bsertitle{Ast\'erisque},
pp. \bfpage{273}--\blpage{311}
(\byear{1997})
\end{bchapter}
\endbibitem

%%% 14
\bibitem[\protect\citeauthoryear{Goodman and Hartshorne}{1969}]{goodmanSchemesFinitedimensionalCohomology1969}
\begin{barticle}
\bauthor{\bsnm{Goodman}, \binits{J.}},
\bauthor{\bsnm{Hartshorne}, \binits{R.}}:
\batitle{Schemes with finite-dimensional cohomology groups}.
\bjtitle{American Journal of Mathematics}
\bvolume{91},
\bfpage{258}--\blpage{266}
(\byear{1969})
\doiurl{10.2307/2373281}
\end{barticle}
\endbibitem

%%% 15
\bibitem[\protect\citeauthoryear{Hartshorne}{1970}]{hartshorneAmpleSubvarietiesAlgebraic1970}
\begin{bbook}
\bauthor{\bsnm{Hartshorne}, \binits{R.}}:
\bbtitle{Ample Subvarieties of Algebraic Varieties}.
\bsertitle{Lecture {{Notes}} in {{Mathematics}}},
vol. \bseriesno{Vol. 156}.
\bpublisher{Springer},
\blocation{Berlin-New York}
(\byear{1970})
\end{bbook}
\endbibitem

%%% 16
\bibitem[\protect\citeauthoryear{Hironaka and Matsumura}{1968}]{hironakaFormalFunctionsFormal1968}
\begin{barticle}
\bauthor{\bsnm{Hironaka}, \binits{H.}},
\bauthor{\bsnm{Matsumura}, \binits{H.}}:
\batitle{Formal functions and formal embeddings}.
\bjtitle{Journal of the Mathematical Society of Japan}
\bvolume{20},
\bfpage{52}--\blpage{82}
(\byear{1968})
\doiurl{10.2969/jmsj/02010052}
\end{barticle}
\endbibitem

%%% 17
\bibitem[\protect\citeauthoryear{Dao and Takagi}{2016}]{daoRelationshipDepthCohomological2016}
\begin{barticle}
\bauthor{\bsnm{Dao}, \binits{H.}},
\bauthor{\bsnm{Takagi}, \binits{S.}}:
\batitle{On the relationship between depth and cohomological dimension}.
\bjtitle{Compositio Mathematica}
\bvolume{152}(\bissue{4}),
\bfpage{876}--\blpage{888}
(\byear{2016})
\doiurl{10.1112/S0010437X15007678}
\end{barticle}
\endbibitem

%%% 18
\bibitem[\protect\citeauthoryear{Hern{\'a}ndez et~al.}{2018}]{hernandezCohomologicalDimensionLyubeznik2018}
\begin{barticle}
\bauthor{\bsnm{Hern{\'a}ndez}, \binits{D.J.}},
\bauthor{\bsnm{{N{\'u}{\~n}ez-Betancourt}}, \binits{L.}},
\bauthor{\bsnm{P{\'e}rez}, \binits{F.}},
\bauthor{\bsnm{Witt}, \binits{E.E.}}:
\batitle{Cohomological dimension, {{Lyubeznik}} numbers, and connectedness in mixed characteristic}.
\bjtitle{Journal of Algebra}
\bvolume{514},
\bfpage{442}--\blpage{467}
(\byear{2018})
\doiurl{10.1016/j.jalgebra.2018.07.019}
\end{barticle}
\endbibitem

%%% 19
\bibitem[\protect\citeauthoryear{Faltings}{1980}]{faltingsUberLokaleKohomologiegruppen1980}
\begin{barticle}
\bauthor{\bsnm{Faltings}, \binits{G.}}:
\batitle{\"uber lokale {{Kohomologiegruppen}} hoher {{Ordnung}}}.
\bjtitle{Journal f\"ur die Reine und Angewandte Mathematik. [Crelle's Journal]}
\bvolume{313},
\bfpage{43}--\blpage{51}
(\byear{1980})
\doiurl{10.1515/crll.1980.313.43}
\end{barticle}
\endbibitem

\end{thebibliography}

\end{document}